\documentclass[11pt]{article}
\usepackage{mathrsfs}
\usepackage{bbm}
\usepackage{amsfonts}
\usepackage{amsmath,amssymb}
\usepackage{amsmath}
\usepackage{amssymb}
\usepackage{graphicx}

\newtheorem{theorem}{Theorem}[section]
\newtheorem{corollary}[theorem]{Corollary}
\newtheorem{definition}[theorem]{Definition}

\newtheorem{lemma}[theorem]{Lemma}
\newtheorem{proposition}[theorem]{Proposition}
\newtheorem{remark}[theorem]{Remark}

\newenvironment{proof}[1][Proof]{\noindent \textbf{#1.} }{\  $\Box$}
\numberwithin{equation}{section}

\begin{document}

\title{\textbf{Gradient Estimates for  Nonlinear Diffusion Semigroups by Coupling Methods }}
\author{ Yongsheng Song\thanks{Academy of Mathematics
and Systems Science, CAS, Beijing, China, yssong@amss.ac.cn.
Research supported by by NCMIS; Youth Grant of NSF (No. 11101406);
Key Project of NSF (No. 11231005); Key Lab of Random Complex
Structures and Data Science, CAS (No. 2008DP173182).} }
\maketitle
\date{}

\begin{abstract}Our purpose is to obtain gradient estimates for certain nonlinear partial differential equations by coupling methods.  First we derive uniform gradient estimates
for a certain semi-linear PDEs based on the coupling method
introduced in Wang (2011) and the theory of backward SDEs. Then we
generalize Wang's coupling to the $G$-expectation space and obtain
gradient estimates for  nonlinear diffusion semigroups, which
correspond to the solutions of a certain fully nonlinear PDEs.

\end{abstract}


\textbf{MSC-classification}: 35K55, 60H10, 60J60

%%% ----------------------------------------------------------------------
%%% ----------------------------------------------------------------------

\section{Introduction}
The classical Feynman-Kac formula established the connection between
linear partial differential equations and stochastic processes,
which is the starting point for the study of PDEs by probabilistic
methods. For example, probabilistic tools, such as Malliavin calculus,
the method of coupling, etc, can be used to obtain the regularity
property for PDEs.

In 1990, Pardoux and Peng (\cite {PP90}) introduced the general
nonlinear BSDEs, based on which \cite {Peng1991} and \cite {PP92}
gave a probabilistic formula for a certain quasi-linear parabolic
partial differential equations. This is the so-called generalized
Feynman-Kac formula, which provides a way  to study quasi-linear
PDEs by probabilistic methods. \cite {FT} established
Bismut-Elworthy formula for backward SDEs by the method of Malliavin
caiculus. As in the linear case, this formula provides  gradient
estimates for the solutions to the associated PDEs.

A natural question is how to give a probabilistic interpretation for
fully nonlinear PDEs, which is one of the main motivations
for Shige Peng to establish the fully nonlinear expectation theory.
$G$-expectation is a typical time-consistent sublinear expectation.
 In the $G$-expectation space, the nonlinear semigroups associated
 with SDEs driven by $G$-Brownian motion are solutions to certain
 fully nonlinear PDEs.

 In this article, we derive gradient estimates for certain nonlinear
 partial differential equations by coupling methods.  First, in Section 3, we obtain
 uniform gradient estimates for a certain semi-linear PDEs based on
 the coupling method introduced in \cite {Wang11} and the theory of backward SDEs. Then, in Section 4, we
generalize Wang's coupling to the $G$-expectation space and obtain
gradient estimates for  nonlinear diffusion semigroups. Our main
results are Theorem \ref {main1} and Theorem \ref {main2}.

\section{Some Definitions and Notations about $G$-expectation}
We review some basic notions and definitions of the related spaces under $G$%
-expectation. The readers may refer to \cite{P07a}, \cite{P07b}, \cite{P08a}%
, \cite{P08b}, \cite{P10} for more details.

Let $\Omega_T=C_{0}([0,T];\mathbb{R}^{d})$ be the space of all $%
\mathbb{R}^{d}$-valued continuous paths $\omega=(\omega(t))_{t\geq0}\in
\Omega$ with $\omega(0)=0$ and let $B_{t}(\omega)=\omega(t)$ be the
canonical process.

Let us recall the definitions of $G$-Brownian motion and its corresponding $%
G $-expectation introduced in \cite{P07b}. Set
\begin{equation*}
L_{ip}(\Omega_{T}):=\{
\varphi(\omega(t_{1}),\cdots,\omega(t_{n})):t_{1},\cdots,t_{n}\in
\lbrack0,T],\  \varphi \in C_{b,Lip}((\mathbb{R}^{d})^{n}),\ n\in \mathbb{N}%
\},
\end{equation*}
where $C_{b,Lip}(\mathbb{R}^{d})$ is the collection of bounded Lipschitz
functions on $\mathbb{R}^{d}$.

We are given a function%
\begin{equation*}
G:\mathbb{S}_d\mapsto \mathbb{R}
\end{equation*}
satisfying the following monotonicity,  sublinearity and positive homogeneity:

\begin{description}
\item[A1.] \bigskip$G(a)\geq G(b),\  \ $if $a,b\in \mathbb{S}_d$
and $a\geq b;$

\item[A2.] $G(a+b)\leq G(a)+G(b)$, for each $a,b\in \mathbb{S}_d$;

\item[A3.] $  G(\lambda a)=\lambda
G(a)$ for  $a\in \mathbb{S}_d$ and $\lambda \geq0.$
\end{description}

\begin{remark}
When $d=1$, we have $G(a):=\frac{1}{2}(\overline{\sigma}^{2}a^{+}-\underline{%
\sigma}^{2}a^{-})$, for $0\leq \underline{\sigma}^{2}\leq \overline{\sigma}%
^{2}$.
\end{remark}

\bigskip \ For each $\xi(\omega)\in L_{ip}(\Omega_{T})$ of the form
\begin{equation*}
\xi(\omega)=\varphi(\omega(t_{1}),\omega(t_{2}),\cdots,\omega(t_{n})),\  \
0=t_{0}<t_{1}<\cdots<t_{n}=T,
\end{equation*}
we define the following conditional  $G$-expectation
\begin{equation*}
\mathbb{E}_{t}[\xi]:=u_{k}(t,\omega(t);\omega(t_{1}),\cdots,\omega
(t_{k-1}))
\end{equation*}
for each $t\in \lbrack t_{k-1},t_{k})$, $k=1,\cdots,n$. Here, for each $%
k=1,\cdots,n$, $u_{k}=u_{k}(t,x;x_{1},\cdots,x_{k-1})$ is a function of $%
(t,x)$ parameterized by $(x_{1},\cdots,x_{k-1})\in (\mathbb{R}^d)^{k-1}$, which
is the solution of the following PDE ($G$-heat equation) defined on $%
[t_{k-1},t_{k})\times \mathbb{R}^d$:
\begin{equation*}
\partial_{t}u_{k}+G(\partial^2_{x}u_{k})=0\
\end{equation*}
with terminal conditions
\begin{equation*}
u_{k}(t_{k},x;x_{1},\cdots,x_{k-1})=u_{k+1}(t_{k},x;x_{1},\cdots x_{k-1},x),
\, \, \hbox{for $k<n$}
\end{equation*}
and $u_{n}(t_{n},x;x_{1},\cdots,x_{n-1})=\varphi (x_{1},\cdots x_{n-1},x)$.

The $G$-expectation of $\xi(\omega)$ is defined by $\mathbb{E}[\xi]=%
\mathbb{E}_{0}[\xi]$. From this construction we obtain a natural norm $%
\left \Vert \xi \right \Vert _{L_{G}^{p}}:=\mathbb{E}[|\xi|^{p}]^{1/p}$.
The completion of $L_{ip}(\Omega_{T})$ under $\left \Vert \cdot \right \Vert
_{L_{G}^{p}}$ is a Banach space, denoted by $L_{G}^{p}(\Omega_{T})$. The
canonical process $B_{t}(\omega):=\omega(t)$, $t\geq0$, is called a $G$%
-Brownian motion in this sublinear expectation space $(\Omega,L_{G}^{1}(%
\Omega ),\mathbb{E})$.

\begin {definition} A process $\{M_t\}$ with values in
$L^1_G(\Omega_T)$ is called a $G$-martingale if $\mathbb{E}_s(M_t)=M_s$
for any $s\leq t$. If $\{M_t\}$ and  $\{-M_t\}$ are both
$G$-martingales, we call $\{M_t\}$ a symmetric $G$-martingale.
\end {definition}

\begin{theorem}
\label{the2.7} (\cite{DHP11,HP09}) There exists a weakly compact subset $\mathcal{P}%
\subset\mathcal{M}_{1}(\Omega_{T})$, the set of probability measures on
$(\Omega_{T},\mathcal{B}(\Omega_{T}))$, such that
\[
\mathbb{E}[\xi]=\sup_{P\in\mathcal{P}}E_{P}[\xi]\ \ \text{for
\ all}\ \xi\in L_{ip}(\Omega_T).
\]
$\mathcal{P}$ is called a set that represents $\mathbb{E}$.
\end{theorem}
Let $\mathcal{P}$ be a weakly compact set that represents $\mathbb{E}$.
For this $\mathcal{P}$, we define capacity%
\[
c(A):=\sup_{P\in\mathcal{P}}P(A),\ A\in\mathcal{B}(\Omega_{T}).
\] $c$ defined here is independent of the choice of $\mathcal{P}$ (see Remark 2.7 in \cite{Song11} for details).
We say a set $A\subset\Omega_{T}$ is polar if $c(A)=0$. A property holds
\textquotedblleft quasi-surely\textquotedblright\ (q.s. for short) if it holds
outside a polar set.

\begin{definition}
A function $\eta (t,\omega ):[0,T]\times \Omega _{T}\rightarrow \mathbb{R}$
is called a step process if there exists a time partition $%
\{t_{i}\}_{i=0}^{n}$ with $0=t_{0}<t_{1}<\cdot \cdot \cdot <t_{n}=T$, such
that for each $k=0,1,\cdot \cdot \cdot ,n-1$ and $t\in (t_{k},t_{k+1}]$
\begin{equation*}
\eta (t,\omega )=\xi_{t_k}\in L_{ip}(\Omega_{t_k}).
\end{equation*}%
 We denote by $M^{0}(0,T)$ the
collection of all step processes.
\end{definition}

 For a step process $\eta \in
M^{0}(0,T) $, we set the norm
\begin{equation*}
\Vert \eta \Vert _{H_{G}^{p}}^{p}:=\mathbb{E}[\{ \int_{0}^{T}|\eta
_{s}|^{2}ds\}^{p/2}],\  \  \  \Vert \eta \Vert _{M_{G}^{p}}^{p}:=\mathbb{E}[\int_{0}^{T}|\eta _{s}|^{p}ds]\  \
\end{equation*}
and denote by $H_{G}^{p}(0,T)$ and $M_{G}^{p}(0,T)$ the completion of $%
M^{0}(0,T)$ with respect to the norms $\Vert \cdot \Vert _{H_{G}^{p}}$ and $%
\Vert \cdot \Vert _{M_{G}^{p}}$, respectively.

\begin {definition}

(i)We say that a map $\xi(\omega): \Omega_T\rightarrow \mathbb{R}$  is quasi-continuous if for all $\varepsilon>0$, there exists an open set
$G$ with $c(G)<\varepsilon$ such that $\xi(\cdot)$ is continuous
on $G^c$.

(ii) We say that a process $M_t(\omega): \Omega_T\times [0,T]\rightarrow \mathbb{R}$  is quasi-continuous if for all $\varepsilon>0$, there exists an open set
$G$ with $c(G)<\varepsilon$ such that $M_\cdot(\cdot)$ is continuous
on $G^c\times[0, T]$.
\end {definition}
\begin {remark}i) Different from  Wiener probability space, counterexamples can be given to show that not all   Borel measurable functions on $\Omega_T$ are $c$-quasi continuous.

ii) For $\eta\in M^1_G(0,T)$, it's easy to see that the process $\int_0^t\eta_s(\omega)ds$ has a $c$-quasi continuous version; Also, \cite {Song11} shows that any $G$-martingale has a $c$-quasi continuous version.
\end {remark}
\begin {theorem} (\cite{DHP11})  For $p\geq1$,
\[L^p_G(\Omega_T)=\{X\in L^0: X \ \textmd{has a} \ q.c. \ version, \
\lim_{n\rightarrow\infty}\mathbb{E}[|X|^p1_{\{|X|>n\}}]=0\},\] where
$L^0$ denotes the space of all $\mathbb{R}$-valued Borel measurable functions on
$\Omega_T$.
\end {theorem}

\section{Gradient Estimates for Quasi-linear PDEs}
We consider the forward-backward stochastic differential equations  below
\begin {eqnarray}\label {pde-sl} & & X^x_t=x+\int_0^t\sigma(X^x_s)d B_s +\int_0^t b(X^x_s)ds, \\
                 \label {pde-sl} & & Y^x_t=\varphi(X^x_T)+\int_t^Tg(Y^x_s, Z^x_s)ds-\int_t^TZ^x_sd B_s,
\end {eqnarray} where $B$ is a $d$-dimensional standard Brownian motion. Set $u(T,x)=Y^x_0$.  By the generalized Feynman-Kac formula given in \cite {Peng1991} and \cite {PP92}, $u$ is the solution to the following PDE
\begin {eqnarray}  \partial_t u-\mathcal{L}u-g(u, \sigma^*Du)&=&0,\\
                                                         u(0,x)&=&\varphi(x),
\end {eqnarray}
where the generator $$\mathcal{L}f=\frac{1}{2}\mathrm{tr}[\sigma\sigma^*D^2f]+b\cdot Df.$$

\

Hypothesis ($\textbf{P}$).

(i) $\sigma: \mathbb{R}^d\rightarrow \mathbb{S}_d$, $b: \mathbb{R}^d\rightarrow \mathbb{R}^d$ are  Lipschitz continuous.
\begin {eqnarray*}& &  |\sigma(x)-\sigma(x')|\leq L_\sigma |x-x'|;\\
  & & |b(x)-b(x')|\leq L_b |x-x'|;
\end {eqnarray*}

(ii) There exists $\Lambda_\sigma\geq \lambda_\sigma>0$ such that
$\lambda_\sigma I \leq\sigma(x)\leq \Lambda_\sigma I$;

(iii) There exists $K_g, L_g>0$ such that $|g(y,z)-g(y',z')|\leq K_g|y-y'|+L_g|z-z'|$.

For convenience, we denote by
$\beta_\sigma:=\frac{\Lambda_\sigma}{\lambda_\sigma}$ and $g_0=g(0,0)$.

\

\

Below is the main result in this section.

\begin {theorem} \label {main1} Assume that Hypothesis (\textbf{P}) holds. Let $u(T,x)=Y^x_0$. There exists a constant $C>0$ depending on $\Lambda_\sigma, \lambda_\sigma, K_g, L_g, g_0$ such that
\[|u(T,x)-u(T,y)|\leq C(\|\varphi\|_\infty+|g_0|/\mu)\frac{e^{\mu T}}{\sqrt{(1-e^{-LT})/L}}|x-y|,\] where $\mu=K_g+4L_g^2, \ L=2(L_g L_\sigma+L_b+ 2 L_\sigma^2)$.
\end {theorem}
\begin {remark} The constant $C$ above can be chosen as $C=4C_5C_{\beta_\sigma}C_g\frac{\Lambda_\sigma^2}{\lambda_\sigma^3}$
with
$C_5:=(c_{\frac{5}{2}})^{\frac{2}{5}}(\frac{1}{5})^{\frac{1}{5}}(\frac{4}{5})^{\frac{4}{5}}$,
$C_{\beta_\sigma}=\frac{1}{\sqrt{2}}d_{32(\beta_\sigma^3+\frac{1}{4})^2}+1$,
$C_g=1+\frac{K_g}{L_g^2}$, where $c_p,
d_p$ are   constants depending only on  $p$ in Lemma \ref
{esti2-drif-l} and Lemma \ref {esti-bsde}, respectively.
\end {remark}

\begin {corollary} Assume that Hypothesis (\textbf{P}) holds. Let $u(T,x):=E[\varphi(X^x_T)]$.  There exists a constant $C>0$ depending on $\Lambda_\sigma, \lambda_\sigma$ such that
\[|u(T,x)-u(T,y)|\leq C\|\varphi\|_\infty\frac{1}{\sqrt{(1-e^{-LT})/L}}|x-y|,\] where $ L=2L_b+ 4 L_\sigma^2$.
\end {corollary}
\begin {proof} This is a spacial case of Theorem \ref {main1} with $g\equiv0$. So we can assume $g_0=0$, $K_g=0$ and $L_g=\frac{1}{n}$, which implies that $C_g=1$. So
\[|u(T,x)-u(T,y)|\leq 4C_5C_{\beta_\sigma}\frac{\Lambda_\sigma^2}{\lambda_\sigma^3}\|\varphi\|_\infty\frac{e^{\mu T}}{\sqrt{(1-e^{-LT})/L}}|x-y|,\] where $\mu=\frac{4}{n^2}, \ L=2(\frac{L_\sigma}{n}+L_b+ 2 L_\sigma^2)$. Letting $n$ go to infinity, we get the desired result.

\end {proof}

\subsection {Construction of the Coupling}

Let $\xi_t:=\frac{\alpha}{\alpha-1}\frac{\theta}{L}(1-e^{
L(t-T)})=:\frac{\alpha\theta}{\alpha-1}\xi^0_t$ for some
$\alpha\geq2$, where
\begin {eqnarray*}
\theta=\frac{\lambda^2_\sigma\Lambda^{-1}_\sigma}{2}, \ L&=&2
L^g_{\alpha/2}, \ L^g_{\alpha/2}= L_g L_\sigma+L_b+
\frac{\alpha-1}{2} L_\sigma^2.
\end {eqnarray*}
\subsubsection {the Coupling before time $T$}
 Consider the coupling below which is adapted from \cite {Wang11}
\begin {eqnarray}\label {l-equX} d X^x_t&=&\sigma(X^x_t)d B_t +b(X^x_t)dt, \ X^x_0=x;\\
                 \label {l-equXt}d \tilde{X}^y_t&=&\sigma(\tilde{X}^y_t)d B_t +b(\tilde{X}^y_t)dt+\frac{1}{\xi_t}\sigma(X^x_t)(X^x_t-\tilde{X}^y_t)dt, \  \tilde{X}^y_0=y.
\end {eqnarray}  Let $\varepsilon_t$ be an adapted measurable process such that $|\varepsilon_t|\leq L_g$. Set \[B^\varepsilon_t=B_t-\int_0^t\varepsilon_sds.\] Rewrite (\ref{l-equX}-\ref{l-equXt})
\begin {eqnarray}\label {l-equXr} d X^x_t&=&\sigma(X^x_t)d B^\varepsilon_t +b^\varepsilon(t,X^x_t)dt, \ X^x_0=x;\\
                 \label {l-equXtr}d \tilde{X}^y_t&=&\sigma(\tilde{X}^y_t)d B^\varepsilon_t +b^\varepsilon(t,\tilde{X}^y_t)dt+\frac{1}{\xi_t}\sigma(X^x_t)(X^x_t-\tilde{X}^y_t)dt, \  \tilde{X}^y_0=y,
\end {eqnarray} where $b^\varepsilon(t,x)=\sigma(x)\varepsilon_t+b(x)$.

Then $\hat{X}_t:=X^x_t-\tilde{X}^y_t$ satisfies the equation below
\begin {eqnarray}\label {l-equXhr} d \hat{X}_t= \hat{\sigma}(t)d B^\varepsilon_t +\hat{b}^\varepsilon(t)dt-
\frac{1}{\xi_t}\sigma(X^x_t)\hat{X}_t dt, \ \hat{X}_0=x-y.
\end {eqnarray} where $\hat{\sigma}(t)=\sigma(X^x_t)-\sigma(\tilde{X}^y_t)$, $\hat{b}^\varepsilon(t)=b^\varepsilon(t,X^x_t)-b^\varepsilon(t,\tilde{X}^y_t)$.

Set $\tilde{B}^\varepsilon_t=B^\varepsilon_t+\int_0^t\frac{\sigma(\tilde{X}^y_s)^{-1}\sigma(X^x_s)\hat{X}_s}{\xi_s}ds$. Rewrite equations (\ref {l-equXr}-\ref {l-equXhr}) as
\begin {eqnarray}
\label {l-equX2r} d X^x_t&=&\sigma(X^x_t)d \tilde{B}^\varepsilon_t +b^\varepsilon(t, X^x_t)dt-\frac{1}{\xi_t}\sigma(X^x_t)\sigma(\tilde{X}^y_t)^{-1}\sigma(X^x_t)\hat{X}_t dt, \ X^x_0=x;\\
\label {l-equXt2r}d \tilde{X}^y_t&=&\sigma(\tilde{X}^y_t)d \tilde{B}^\varepsilon_t +b^\varepsilon(t,\tilde{X}^y_t)dt, \  \tilde{X}^y_0=y;\\
\label {l-equXh2r} d \hat{X}_t&=&\hat{\sigma}(t)d \tilde{B}^\varepsilon_t +\hat{b}^\varepsilon(t)dt-\frac{1}{\xi_t}\sigma(X^x_t)\sigma(\tilde{X}^y_t)^{-1}\sigma(X^x_t)\hat{X}_t dt, \ \hat{X}_0=x-y.
\end {eqnarray}

By It\^o's formula, we have
\begin {eqnarray*}H_t:=|\hat{X}_t|^2=& &|x-y|^2+\int_0^t2\hat{\sigma}(s)^* \hat{X}_s\cdot d B^\varepsilon_s+\int_0^t2\hat{X}_s\cdot\hat{b}^\varepsilon(s)ds \\
& &-\int_0^t\frac{2\hat{X}_s^*\sigma(X^x_s)\hat{X}_s}{\xi_s}ds+\int_0^t \mathrm{tr}[\hat{\sigma}(s)^*\hat{\sigma}(s)]d s,\\
=& &|x-y|^2+\int_0^t2\hat{\sigma}(s)^* \hat{X}_s\cdot d \tilde{B}^\varepsilon_s+\int_0^t2\hat{X}_s\cdot\hat{b}^\varepsilon(s)ds \\
& & -\int_0^t\frac{2\hat{X}_s^*\sigma(X^x_t)\sigma(\tilde{X}^y_t)^{-1}\sigma(X^x_s)\hat{X}_s}{\xi_s}ds+\int_0^t \mathrm{tr}[\hat{\sigma}(s)^*\hat{\sigma}(s)]d s, \ t\in (0,T).
\end {eqnarray*} So, setting $L_{b^\varepsilon}:=L_gL_\sigma+L_b$, we have
\begin {eqnarray*} & & d H_t\leq 2\hat{\sigma}(t)^* \hat{X}_t\cdot d B^\varepsilon_t +(2L_{b^\varepsilon}+L_\sigma^2-\frac{2\lambda_\sigma}{\xi_t})H_tdt, \  t\in (0,T),\\
                   & & d H_t\leq 2\hat{\sigma}(t)^* \hat{X}_t\cdot d \tilde{B}^\varepsilon_t +(2L_{b^\varepsilon}+L_\sigma^2-\frac{2\lambda^2_\sigma \Lambda^{-1}_\sigma}{\xi_t})H_tdt, \ t\in (0,T),
\end {eqnarray*}
and, for $p\geq1$,
\begin {eqnarray*}
d H_t^p&\leq&2pH^{p-1}_t\hat{\sigma}(t)^* \hat{X}_t\cdot d B^\varepsilon_t+(2L_{b^\varepsilon}+(2p-1) L_\sigma^2-\frac{2\lambda_\sigma}{\xi_t})pH^p_t dt\\
       &=&:2pH^{p-1}_t\hat{\sigma}(t)^* \hat{X}_t\cdot d B^\varepsilon_t+(2L^g_p-\frac{2\lambda_\sigma}{\xi_t})p H^p_t dt.\\
d H_t^p&\leq&2pH^{p-1}_t\hat{\sigma}(t)^* \hat{X}_t\cdot d \tilde{B}^\varepsilon_t+(2L^g_p-\frac{2\lambda^2_\sigma\Lambda^{-1}_\sigma}{\xi_t})p H^p_t dt.
\end {eqnarray*} where $L^g_p=L_{b^\varepsilon}+(p-\frac{1}{2}) L_\sigma^2$. Then, setting $\tilde{H}_t=e^{2\varrho t}H_t$,
\begin {eqnarray}
 \label {lp-inequ-l}\frac{\tilde{H}^p_t}{\xi_t^{2p-1}}
&\leq&\frac{\tilde{H}^p_s}{\xi_s^{2p-1}}+\int_s^t\frac{2pe^{2p\varrho r}}{\xi^{2p-1}_r}H^{p-1}_r\hat{\sigma}(r)^* \hat{X}_r\cdot d B^\varepsilon_r\\
\label {lp-inequ2-l}& &+\int_s^t\frac{2p\tilde{H}^p_r}{\xi^{2p}_r}((\varrho+L^g_p)\xi_r-\lambda_\sigma-\frac{2p-1}{2p}\xi'_r)dr,\\
 \label {lp-INEQU-l}\frac{\tilde{H}^p_t}{\xi_t^{2p-1}}
&\leq&\frac{\tilde{H}^p_s}{\xi_s^{2p-1}}+\int_s^t\frac{2pe^{2p\varrho r}}{\xi^{2p-1}_r}H^{p-1}_r\hat{\sigma}(r)^* \hat{X}_r\cdot d \tilde{B}^\varepsilon_r\\
\label {lp-INEQU2-l} & &+\int_s^t\frac{2p\tilde{H}^p_r}{\xi^{2p}_r}((\varrho+L^g_p)\xi_r-
\frac{\lambda^2_\sigma}{\Lambda_\sigma}-\frac{2p-1}{2p}\xi'_r)dr.
\end {eqnarray}
\label {sec-l-1.2}\subsubsection {Extension of the Coupling to $T$}
Setting $\varrho=\frac{p-1}{p}L^g_p$, we have $\varrho+L^g_p=\frac{2p-1}{p}L^g_p$. By the definition of $\xi$, for any $\frac{\alpha}{2}\geq p\geq 1$, we have
\begin {eqnarray}\label {c1}\frac{2p-1}{p}L^g_p\xi_r-\lambda_\sigma-\frac{2p-1}{2p}\xi'_r\leq \frac{2p-1}{p}L^g_p\xi_r-\frac{\lambda^2_\sigma}{\Lambda_\sigma}-\frac{2p-1}{2p}\xi'_r\leq-\theta.
\end {eqnarray}
So, by (\ref {lp-inequ-l}-\ref {lp-inequ2-l}), we have
\begin {eqnarray*} E^\varepsilon[\int_0^t\frac{e^{2p\varrho r}|\hat{X}_r|^{2p}}{\xi^{2p}_r}dr]\leq \frac{1}{2p\theta} \frac{|x-y|^{2p}    }{\xi_{0}^{2p-1}},
\end {eqnarray*} where $E^\varepsilon$ is the expectation under $P^\varepsilon$ with $\frac{dP^\varepsilon}{dP}:=e^{\int_0^T\varepsilon_sdB_s-\frac{1}{2}\int_0^T|\varepsilon_s|^2ds}$.

Actually, for any $p\geq 1$, there exists $s(p)\in(0,T)$ such that
$\frac{2p-1}{p}L^g_p\xi_r-\lambda_\sigma-\frac{2p-1}{2p}\xi'_r\leq-\frac{\theta}{2p}$
for any $s\in [s(p),T)$. So
$$E^\varepsilon[\int_{s(p)}^t\frac{e^{2p\varrho r}|\hat{X}_r|^{2p}}{\xi^{2p}_r}]dr\leq
\frac{1}{\theta} E^\varepsilon[\frac{|e^{2p\varrho s(p)}\hat{X}_{s(p)}|^{2p}
}{\xi_{s(p)}^{2p-1}}]<\infty.$$ Consequently, for any $p\geq1$,
there exists $C(p)>0$ such that
    $$E^\varepsilon[\int_0^t\frac{|\hat{X}_r|^{p}}{\xi^{p}_r}]dr\leq C(p).$$
 So there exists a $\mathbb{R}^d$-valued adapted measurable process $\{g_t\}_{t\in[0,T]}$ such that
\begin {eqnarray}\label {Drift} E^\varepsilon[\int_0^T|g_s|^pds]<\infty \ \textmd{and}  \ g_s=\frac{1}{\xi_s}(X^x_s-\tilde{X}^y_s), \ s\in[0,T).
 \end {eqnarray}

 Let $(\bar{X}_t)_{t\in[0,T]}$ be the solution to the equation below
\begin {eqnarray}
                 \label {l-equXte}d \bar{X}_t&=&\sigma(\bar{X}_t)d B_t +b(\bar{X}_t)dt+\sigma(X^x_t)g_t dt, \  \bar{X}_0=y.
\end {eqnarray} Clearly, we have $\tilde{X}^y_t=\bar{X}_t$, $t\in[0,T)$. So $\bar{X}$ is a continuous extension of $\tilde{X}^y$ to $[0,T]$. In the sequel, we shall still write $\tilde{X}^y$ for $\bar{X}$.

\begin {proposition} \label {extention} Let $X^x$ and $\tilde{X}^y$ be the solutions to equations (\ref {l-equX}) and (\ref {l-equXte}). We have $X^x_T=\tilde{X}^y_T$.
\end {proposition}
\begin {proof} For $\omega$ such that $X^x_T(\omega)\neq \tilde{X}^y_T(\omega)$, we have $$\int_0^T|g_s(\omega)|^pds=\int_0^T\frac{1}{\xi^p_s}|X^x_s(\omega)-\tilde{X}^y_s(\omega)|^pds=\infty.$$ Noting that  $E^\varepsilon[\int_0^T|g_s|^pds]\leq C(p),$ we conclude that $X^x_T=\tilde{X}^y_T$ a.s.
\end {proof}

\subsubsection{Estimates for the Drift}

Set $h_s=-\frac{1}{\xi_s}\sigma(\tilde{X}^y_s)^{-1}\sigma(X^x_s)(X^x_s-\tilde{X}^y_s)$. Choose a sequence of  stopping times $\tau_n$ such that $\tau_n\uparrow T$, $\tau_n\leq T-\frac{1}{n}$ and $h^n_s:=h_s1_{[0,\tau_n]}$ is  bounded. Denote by $\tilde{E}^n$ the expectation under the probability $\tilde{P}^n$ with $\frac{d\tilde{P}^n}{dP^\varepsilon}=e^{\int_0^Th^ndB^\varepsilon_s-\frac{1}{2}\int_0^T|h^n_s|^2ds}=: U^{h^n}_T$.

\begin {proposition}\label {esti-drift-l} \[\widetilde{E}^{n}[\exp\{\frac{\theta^2}{8\Lambda_\sigma^2}\int_0^{\tau_n}\frac{|\hat{X}_s|^2}{\xi_s^2}ds\}]
\leq \exp\{\frac{\theta}{8\Lambda_\sigma^2\xi_0}|x-y|^2\}.\]
\[ E^\varepsilon[| U^{h^n}_T|^{1+\delta}]\leq \exp\{\frac{\theta\sqrt{1+\delta^{-1}}}{8\Lambda_\sigma^2\xi_0(1+\sqrt{1+\delta^{-1}})}|x-y|^2\},\]
 where $\delta:=\frac{\theta^2}{4\Lambda_\sigma^2\beta_\sigma^2+4\theta \Lambda_\sigma\beta_\sigma}=\frac{1}{16\beta_\sigma^6+8\beta_\sigma^3}$.
\end {proposition}

The estimates above are from Lemma 2.2 in \cite {Wang11}. For readers' convenience, we give the sketch of the proof.

\begin {proof} For $p=1$, (\ref {lp-INEQU-l}-\ref {lp-INEQU2-l}) with $\varrho=0$ shows that
\[\int_0^t\frac{|\hat{X}_r|^2}{\xi_r^2}dr\leq \frac{|x-y|^2}{2\theta\xi_0}+\int_0^t\frac{1}{\theta\xi_r}\hat{\sigma}(r)^* \hat{X}_r\cdot d \tilde{B}^\varepsilon_r.\]
Set $\tilde{B}^n_t:=B^\varepsilon_t-\int_0^th^n_sds$. By Girsanov transformation, we know that $\tilde{B}^n_t$ is a standard Brownian motion under $\tilde{E}^{n}$. Noting that\[\int_0^{\tau_n}\frac{|\hat{X}_r|^2}{\xi_r^2}dr\leq \frac{|x-y|^2}{2\theta\xi_0}+\int_0^{\tau_n}\frac{1}{\theta\xi_r}\hat{\sigma}(r)^* \hat{X}_r\cdot d \tilde{B}^n_r,\] we get
\[\tilde{E}^{n}[\exp\{a\int_0^{\tau_n}\frac{|\hat{X}_r|^2}{\xi_r^2}dr\}]\leq \exp\{\frac{a|x-y|^2}{2\theta\xi_0}\} (\tilde{E}^{n}[\exp\{\frac{8a^2\Lambda_\sigma^2}{\theta^2}\int_0^{\tau_n}\frac{|\hat{X}_r|^2}{\xi_r^2}dr\}])^{1/2} .\] Taking $a=\frac{\theta^2}{8\Lambda_\sigma^2}$, we get the first estimate.

By the definition of $ U^{h^n}_T$, we have
\begin {eqnarray*}E^\varepsilon[| U^{h^n}_T|^{1+\delta}]=\tilde{E}^{n}[\exp\{\delta\int_{0}^{T}h^n_{s}\cdot dB^\varepsilon_{s}-\frac{\delta}{2}\int_{0}^{T}|h^n_s|^2
ds\}].
\end {eqnarray*}Noting that $M_t:=\int_{0}^{t} h^n_{s}\cdot dB^\varepsilon_{s}-\int_{0}^{t}|h^n_s|^2ds$ is a  martingale under $\tilde{P}^{n}$, we have, for any $q>1$,
\begin {eqnarray*}E^\varepsilon[| U^{h^n}_{T}|^{1+\delta}]
&=&\tilde{E}^{n}[\exp \{\delta M_T+\frac{\delta}{2}\langle M\rangle_T\}]\\
&\leq&(\tilde{E}^{n}[\exp\{\frac{\delta q(\delta q+1)}{2(q-1)}\langle M\rangle_T\}])^{(q-1)/q}\\
&\leq&(\tilde{E}^{n}[\exp\{\frac{\delta q(\delta q+1)}{2(q-1)}\beta_\sigma^2\int_0^{\tau_n}\frac{|\hat{X}_s|^2}{\xi_s^2}ds\}])^{(q-1)/q}.
\end {eqnarray*} Taking $q=1+\sqrt{1+\delta^{-1}}$, we get
$\frac{\delta q(\delta q+1)}{2(q-1)}\beta_\sigma^2=\frac{(\delta+\sqrt{\delta^2+\delta})^2\beta_\sigma^2}{2}
=\frac{\theta^2}{8\Lambda_\sigma^2}$. By the first estimate, we get the second desired result.
\end {proof}

So for any adapted measurable process $\varepsilon$ with $|\varepsilon_t|\leq L_g$,
\[\tilde{P}^\varepsilon:= U_T.P^\varepsilon, \ \textmd{with} \ U_T:=\exp\{\int_0^Th_sdB^\varepsilon_s-\frac{1}{2}\int_0^T|h_s|^2ds\}\] is a probability, under which $\tilde{B}^\varepsilon_t=B^\varepsilon_t-\int_0^th_sds$ is a standard Brownian motion. We write $\tilde{P}, \tilde{B}_t$ for $\tilde{P}^0, \tilde{B}^0_t$.

Let $u_t=\int_0^Th_sdB^\varepsilon_s-\frac{1}{2}\int_0^T|h_s|^2ds$.
\begin {lemma}\label {esti2-drif-l}  We have the following
estimates:
\begin {eqnarray}\label {esti-u}& &(E^\varepsilon[|u_T|^{\frac{\alpha}{2}}])^{\frac{2}{\alpha}}\leq C_\alpha\frac{2\Lambda_\sigma^2}{\lambda_\sigma^3}
\frac{|x-y|}{\sqrt{\xi^0_0}}+ O(|x-y|^2),\\
\label {esti-ut}    & &
(\tilde{E}^\varepsilon[|u_T|^{\frac{\alpha}{2}}])^{\frac{2}{\alpha}}\leq
C_\alpha\frac{2\Lambda_\sigma^2}{\lambda_\sigma^3} \frac{|x-y|}{\sqrt{\xi^0_0}}+ O(|x-y|^2),
\end {eqnarray}where $C_\alpha:=(c_{\frac{\alpha}{2}})^{\frac{2}{\alpha}}(\frac{1}{\alpha})^{\frac{1}{\alpha}}(\frac{1}{\alpha^*})^{\frac{1}{\alpha*}}$ and $c_p$ is the   constant for BDG inequalities.
\end {lemma}
\begin {proof}Set $\varrho=\frac{p-1}{p}L^g_p$ for $p\leq\frac{\alpha}{2}$. By (\ref {lp-inequ-l}), (\ref {lp-INEQU-l}) and (\ref {c1}),
\begin {eqnarray}\label {lp-esti-l} & &E^\varepsilon[\int_0^T\frac{e^{2p\varrho r}|\hat{X}_r|^{2p}}{\xi^{2p}_r}dr]\leq \frac{1}{2p\theta} \frac{|x-y|^{2p}    }{\xi_{0}^{2p-1}},\\
& &\label {lp-esti-lt}\tilde{E}^\varepsilon[\int_0^T\frac{e^{2p\varrho r}|\hat{X}_r|^{2p}}{\xi^{2p}_r}dr]\leq \frac{1}{2p\theta} \frac{|x-y|^{2p}    }{\xi_{0}^{2p-1}}.
\end {eqnarray} So $E^\varepsilon[\int_0^Te^{2p\varrho r}|h_r|^{2p}]dr\leq \frac{\beta_\sigma^{2p}}{2p\theta} \frac{|x-y|^{2p}}{\xi_{0}^{2p-1}}.$
Then
\begin {eqnarray*} (E^\varepsilon[|u_T|^p])^{1/p}\leq& & (E^\varepsilon[|\int_0^Th_sdB^\varepsilon_s|^p])^{1/p}+\frac{1}{2}(E[(\int_0^T|h_s|^2ds)^p])^{1/p}\\
\leq& &c_p^{1/p} (E^\varepsilon[(\int_0^T|h_s|^2ds)^{p/2}])^{1/p}+\frac{1}{2}(E^\varepsilon[(\int_0^T|h_s|^2ds)^{p}])^{1/p}\\
\leq& &c_p^{1/p}
\sqrt{(E^\varepsilon[(\int_0^T|h_s|^2ds)^{p}])^{1/p}}+\frac{1}{2}(E^\varepsilon[(\int_0^T|h_s|^2ds)^{p}])^{1/p}.
\end {eqnarray*} For $p=\frac{\alpha}{2}$, we have $\varrho=\frac{\alpha-2}{\alpha}L^g_{\alpha/2}$ and
\begin {eqnarray*}& &(E^\varepsilon[(\int_0^T|h_s|^2ds)^{\alpha/2}])^{1/\alpha}\\
&=&(E^\varepsilon[(\int_0^Te^{-2\varrho s}e^{2\varrho s}|h_s|^2ds)^{\alpha/2}])^{1/\alpha}\\
&\leq&(E^\varepsilon[\int_0^Te^{\alpha\varrho r}|h_r|^{\alpha}]dr)^{1/\alpha} (\frac{\xi_0}{\alpha^*\theta})^\frac{\alpha-2}{2\alpha}\\
&\leq&(\frac{1}{\alpha})^{\frac{1}{\alpha}}(\frac{1}{\alpha^*})^{\frac{1}{\alpha*}}\frac{2\Lambda_\sigma^2}{\lambda_\sigma^3}
\frac{|x-y|}{\sqrt{(1-e^{
-LT})/L}}
 \end {eqnarray*}
 By arguments above, we get
(\ref{esti-u}). (\ref{esti-ut}) can be obtained in the same way.
\end {proof}
\subsection {Gradient Estimates}
For $\varphi\in C_b(\mathbb{R}^d)$, we consider backward SDEs below
\begin {eqnarray}
\label {bsde}& &Y^x_t=\varphi(X^x_T)+\int_t^Tg(Y^x_s, Z^x_s)d s-\int_t^TZ^x_s d B_s,\\
\label {bsdet}& &\tilde{Y}^y_t=\varphi(\tilde{X}^y_T)+\int_t^Tg(\tilde{Y}^y_s, \tilde{Z}^y_s)d s-\int_t^T\tilde{Z}^y_s d \tilde{B}_s.
\end {eqnarray} Set $u(T,x):=Y^x_0$. Clearly, we have $u(T,y):=\tilde{Y}^y_0$. Rewrite (\ref{bsde}-\ref{bsdet})
\begin {eqnarray}
\label {bsder}& &Y^x_t=\varphi(X^x_T)+\int_t^T[g(Y^x_s, Z^x_s)-\varepsilon_s Z^x_s]d s-\int_t^TZ^x_s d B^\varepsilon_s,\\
\label {bsdetr}& &\tilde{Y}^y_t=\varphi(\tilde{X}^y_T)+\int_t^T[g(\tilde{Y}^y_s, \tilde{Z}^y_s)-(\varepsilon_s-h_s) \tilde{Z}^y_s]d s-\int_t^T\tilde{Z}^y_s d B^\varepsilon_s.
\end {eqnarray}
Let $k_s$ be an adapted measurable process  with $|k_t|\leq K_g$. Set $V_t=\exp\{\int_0^tk_sds\}$ and $W_t=U_tV_t$. Applying It\^o's formula, we get
\begin {eqnarray*}
\label {bsdeG}& &V_tY^x_t=V_T\varphi(X^x_T)+\int_t^TV_s[g(Y^x_s, Z^x_s)-k_sY^x_s-\varepsilon_s Z^x_s]ds-\int_t^TV_sZ^x_sdB^\varepsilon_s,\\
\label {bsdetG}& &W_t\tilde{Y}^y_t=W_T\varphi(X^x_T)+\int_t^TW_s[g(\tilde{Y}^y_s, \tilde{Z}^y_s)-k_s\tilde{Y}^y_s-\varepsilon_s \tilde{Z}^y_s]ds-\int_t^TW_s(\tilde{Z}^y_s+\tilde{Y}^y_sh_s)dB^\varepsilon_s.
\end {eqnarray*} So
\begin {eqnarray*}
\hat{Y}_0=& &V_T\varphi(X^x_T)(1-U_T)+(1-U_T)\int_0^TV_s[g(Y^x_s, Z^x_s)-k_sY^x_s-\varepsilon_sZ^x_s]ds\\
& &+\int_0^TW_s[\hat{g}(s)-k_s\hat{Y}_s-\varepsilon_s \hat{Z}_s]ds-\int_0^T[V_sZ^x_s-W_s(\tilde{Z}^y_s+\tilde{Y}^y_sh_s)]dB^\varepsilon_s\\
& &+\int_0^TU_t h_t\int_0^tV_s[g(Y^x_s, Z^x_s)-k_sY^x_s-\varepsilon_s Z^x_s]ds dB^\varepsilon_t,
\end {eqnarray*}where $\hat{Y}_s=Y^x_s-\tilde{Y}^y_s$, $\hat{Z}_s=Z^x_s-\tilde{Z}^y_s$, $\hat{g}(s)=g(Y^x_s, Z^x_s)-g(\tilde{Y}^y_s, \tilde{Z}^y_s)$. Choosing processes $k^0, \varepsilon^0$ such that $\hat{g}(s)-k^0_s\hat{Y}_s-\varepsilon^0_s \hat{Z}_s\leq0$, we have
\begin {eqnarray}\label {Grad-esti}\hat{Y}_0&\leq& E^{\varepsilon^0}[(1-U_T)[V_T\varphi(X^x_T)+\int_0^TV_s[g(Y^x_s, Z^x_s)-k^0_sY^x_s-\varepsilon^0_sZ^x_s]ds]]\\
\label {Grad-esti2}&\leq& e^{K_gT}\|\varphi\|_\infty E^{\varepsilon^0}[|1-U_T|]+E^{\varepsilon^0}[|1-U_T|\int_0^Te^{K_gs}(2K_g|Y^x_s|+2L_g|Z^x_s|+|g_0|)ds],
\end {eqnarray} where $g_0=g(0,0)$.
\subsubsection{Estimates for the Backward SDEs}
From \cite {BDHPS03}, we have the following estimates:
\begin {proposition} \label {esti-bsde} Let $(Y^x_t, Z^x_t)$ be the solution to the Backward SDE (\ref{bsde}). Set $\mu=K_g+4L_g^2$ Then there exists $d_p>0$ depending on $p, $ such that
\begin {eqnarray*}(E^{\varepsilon^0}[\sup_{t\in[0,T]}e^{\mu pt}|Y^x_t|^p])^{\frac{1}{p}}&\leq& e^{\mu T}d_p(\|\varphi\|_\infty+|g_0|/\mu),\\
(E^{\varepsilon^0}[(\int_0^Te^{2\mu s}|Z^x_s|^2ds)^{p/2}])^{\frac{1}{p}}&\leq& e^{\mu T}d_p(\|\varphi\|_\infty+|g_0|/\mu).
\end {eqnarray*}
\end {proposition}

\subsubsection{Proof to the Gradient Estimates}
\begin {lemma}\label {observation}Let $v: \mathbb{R}^d\rightarrow \mathbb{R}$. Assume that there exists a smooth function $F:\mathbb{R}\rightarrow \mathbb{R}$ such that $F(0)=0$ and
\[|v(x)-v(y)|\leq F(|x-y|), \ \textmd{for any} \ x,y\in \mathbb{R}^d.\] Then \[|v(x)-v(y)|\leq F'(0)|x-y|.\]
\end {lemma}
\begin {proof} Fix $x,y \in \mathbb{R}^d$. Set $x_t=x+t(y-x)$ and $f(t)=v(x_t)$ for $t\in[0,1]$. Then
\[|f(t)-f(s)|\leq F(|t-s||x-y|)=F'(0)|x-y||t-s|+o(|t-s|),\] by which we get the desired result.
\end {proof}

\textbf{Proof to Theorem \ref {main1}.} By Proposition \ref {esti-bsde}, we have the following estimates
\begin {eqnarray*} (e^{K_gT}\|\varphi\|_\infty +|g_0|\int_0^Te^{K_gs}ds)E^{\varepsilon^0}[|U_T-1|]
\leq (\|\varphi\|_\infty+|g_0|/\mu)e^{\mu T}(\|U_T u_T\|_{L^1}+\|u_T\|_{L^1}),
\end {eqnarray*}where $\|\cdot\|_{L^1}$ denotes the norms under $P^{\varepsilon^0}$. Below we shall denote by $\|\cdot\|_{\tilde{L}^p}$ the norms under $\tilde{P}^{\varepsilon^0}$.
\begin {eqnarray*}& &2L_g E^{\varepsilon^0}[|1-U_T|\int_0^Te^{K_gs}|Z^x_s|ds]\\
&\leq& \frac{1}{\sqrt{2}}\|1-U_T\|_{1+\frac{\delta}{2}}(E^{\varepsilon^0}[(\int_0^Te^{2\mu s}|Z^x_s|^2ds)^{\frac{2+\delta}{2\delta}}])^{\frac{\delta}{2+\delta}}\\
&\leq&\frac{1}{\sqrt{2}}d_{\frac{\delta+2}{\delta}}(\|\varphi\|_\infty+|g_0|/\mu)e^{\mu T}
(\|U_T
u_T\|_{L^{1+\frac{\delta}{2}}}+\|u_T\|_{L^{1+\frac{\delta}{2}}});
\end {eqnarray*}
\begin {eqnarray*}& &2K_gE^{\varepsilon^0}[|1-U_T|\int_0^Te^{K_gs}|Y^x_s|ds]\\
&\leq& \frac{K_g}{2L_g^2}E^{\varepsilon^0}[|1-U_T|\sup_{t\in[0,T]}(e^{\mu t}|Y^x_t|)]\\
&\leq&\frac{K_g}{2L_g^2}d_{\frac{\delta+2}{\delta}}(\|\varphi\|_\infty+|g_0|/\mu)e^{\mu T}
(\|U_T
u_T\|_{L^{1+\frac{\delta}{2}}}+\|u_T\|_{L^{1+\frac{\delta}{2}}}).
\end {eqnarray*}
So by (\ref {Grad-esti}-\ref {Grad-esti2}), we have
\[|u(T,x)-u(T,y)|\leq C_{\beta_\sigma}C_g(\|\varphi\|_\infty+|g_0|/\mu)e^{\mu T}
(\|U_T
u_T\|_{L^{1+\frac{\delta}{2}}}+\|u_T\|_{L^{1+\frac{\delta}{2}}}),\] where $C_{\beta_\sigma}=\frac{1}{\sqrt{2}}d_{\frac{\delta+2}{\delta}}+1$  and $C_g=1+\frac{K_g}{L_g^2}$.

By Proposition \ref {esti-drift-l}, we have \[\|U_T
u_T\|_{L^{1+\frac{\delta}{2}}}\leq
(E^{\varepsilon^0}[U_T^{1+\delta}])^{\frac{1}{2+\delta}}\tilde{E}^{\varepsilon^0}[|u_T|^{2+\delta}])^{\frac{1}{2+\delta}}\leq
\exp\{\frac{1}{\Lambda_\sigma^2\xi^0_0}|x-y|^2\}\|u_T\|_{\tilde{L}^{2+\delta}}.\]
Choose $\alpha\geq 4+2\delta$ in the definition of $\xi$. By Lemma
\ref {esti2-drif-l} and Lemma \ref {observation}, we have
\begin {eqnarray*}|u(T,x)-u(T,y)|\leq 4C_\alpha C_{\beta_\sigma}C_g\frac{\Lambda_\sigma^2}{\lambda_\sigma^3}(\|\varphi\|_\infty+|g_0|/\mu)\frac{e^{\mu T}}{\sqrt{\xi^0_0}}.
\end {eqnarray*}
 Setting $\alpha=5$, we get the
desired result.

\section{Gradient Estimates for Nonlinear Semigroups}
In this section, we shall derive  uniform gradient estimates for semigroups associated with stochastic differential equations  below
\begin {eqnarray} d X_t=\sigma(X_t)d B_t +b(X_t)dt,
\end {eqnarray} where $B$ is a $d$-dimensional $G$-Brownian motion with $G(A)=\frac{1}{2}\sup_{\gamma\in\Gamma} \mathrm{tr}(A\gamma)$ for some bounded, closed, convex subset $\Gamma\subset \mathbb{S}_d^+$. The generator of the diffusion process will be $$\mathcal{L}f=G(\sigma^*D^2f\sigma)+b\cdot Df.$$

\

Hypothesis (\textbf{G}).

(i) $\sigma: \mathbb{R}^d\rightarrow \mathbb{S}_d$, $b: \mathbb{R}^d\rightarrow \mathbb{R}^d$ are  Lipschitz continuous.
\begin {eqnarray*}& &  |\sigma(x)-\sigma(x')|\leq L_\sigma |x-x'|;\\
  & & |b(x)-b(x')|\leq L_b |x-x'|;
\end {eqnarray*}

(ii) There exists $\Lambda_\sigma\geq \lambda_\sigma>0$ such that $\lambda_\sigma I \leq\sigma(x)\leq \Lambda_\sigma I$;

(iii) There exists $\Lambda_\Gamma\geq\lambda_\Gamma>0$ such that $\Lambda_\Gamma^2 I\geq\gamma\geq \lambda_\Gamma^2 I$, $\gamma\in \Gamma$.

For convenience, we denote by
$\beta_\sigma:=\frac{\Lambda_\sigma}{\lambda_\sigma}$,
$\beta_\Gamma:=\frac{\Lambda_\Gamma}{\lambda_\Gamma}$.

\

The main result of this section is below.
\begin {theorem}\label {main2}Assume that Hypothesis (\textbf{G}) holds. Let $u(T,x):=\mathbb{E}[\varphi(X^x_T)]$. Then \[|u(T,x)-u(T,y)|\leq C \frac{\|\varphi\|_\infty}{\sqrt{(1-e^{-LT})/L}}|x-y|,\] where
$C=\frac{2\Lambda_\sigma^2}{\lambda_\sigma^3\lambda_\Gamma}$, $L=2L_b+ \Lambda_\Gamma^2 L_\sigma^2$.
\end {theorem}

\subsection {Construction of the Coupling}

First, we shall introduce the construction of the coupling. The arguments below are similar to those in the last section.

Let $\xi_t=\frac{2(\lambda^2_\sigma\Lambda^{-1}_\sigma-\theta)}{L}(1-e^{L(t-T)})$ for some $\theta\in [\lambda^2_\sigma\Lambda^{-1}_\sigma/2, \lambda^2_\sigma\Lambda^{-1}_\sigma)$, where $L=2L_b+ \Lambda_\Gamma^2 L_\sigma^2$. Consider the coupling
\begin {eqnarray}\label {equX} d X_t&=&\sigma(X_t)d B_t +b(X_t)dt, \ X_0=x;\\
                 \label {equY}d Y_t&=&\sigma(Y_t)d B_t +b(Y_t)dt+\frac{1}{\xi_t}\sigma(X_t)(X_t-Y_t)dt, \  Y_0=y.
\end {eqnarray} Then $Z_t:=X_t-Y_t$ satisfies the equation below
\begin {eqnarray}\label {equZ} d Z_t= \hat{\sigma}(t)d B_t +\hat{b}(t)dt-\frac{1}{\xi_t}\sigma(X_t)Z_t dt, \ Z_0=x-y.
\end {eqnarray} where $\hat{\sigma}(t)=\sigma(X_t)-\sigma(Y_t)$, $\hat{b}(t)=b(X_t)-b(Y_t)$.

Set $\tilde{B}_t=B_t+\int_0^t\frac{\sigma(Y_s)^{-1}\sigma(X_s)Z_s}{\xi_s}ds$. Rewrite equations (\ref {equX}-\ref {equZ}) as
\begin {eqnarray}\label {equXr} d X_t&=&\sigma(X_t)d \tilde{B}_t +b(X_t)dt-\frac{1}{\xi_t}\sigma(X_t)\sigma(Y_t)^{-1}\sigma(X_t)Z_t dt, \ X_0=x;\\
                 \label {equYr}d Y_t&=&\sigma(Y_t)d \tilde{B}_t +b(Y_t)dt, \  Y_0=y;\\
                 \label {equZr} d Z_t&=&\hat{\sigma}(t)d \tilde{B}_t +\hat{b}(t)dt-\frac{1}{\xi_t}\sigma(X_t)\sigma(Y_t)^{-1}\sigma(X_t)Z_t dt, \ Z_0=x-y.
\end {eqnarray}

By It\^o's formula, we have
\begin {eqnarray*}|Z_t|^2=& &|x-y|^2+\int_0^t2\hat{\sigma}(s)^* Z_s\cdot d B_s+\int_0^t2Z_s\cdot\hat{b}(s)ds -\int_0^t\frac{2Z_s^*\sigma(X_s)Z_s}{\xi_s}ds\\
& &+\int_0^t \mathrm{tr}[\hat{\sigma}(s)^*\hat{\sigma}(s)d \langle B\rangle_s],\\
=& &|x-y|^2+\int_0^t2\hat{\sigma}(s)^* Z_s\cdot d \tilde{B}_s+\int_0^t2Z_s\cdot\hat{b}(s)ds -\int_0^t\frac{2Z_s^*\sigma(X_t)\sigma(Y_t)^{-1}\sigma(X_s)Z_s}{\xi_s}ds\\
& &+\int_0^t \mathrm{tr}[\hat{\sigma}(s)^*\hat{\sigma}(s)d \langle B\rangle_s], \ t\in (0,T).
\end {eqnarray*} So
\begin {eqnarray*} & & d |Z_t|^2\leq 2\hat{\sigma}(t)^* Z_t\cdot d B_t +(2L_b+L_\sigma^2\Lambda_\Gamma^2-\frac{2\lambda_\sigma}{\xi_t})|Z_t|^2 dt, t\in (0,T),\\
                   & & d |Z_t|^2\leq 2\hat{\sigma}(t)^* Z_t\cdot d \tilde{B}_t +(2L_b+L_\sigma^2\Lambda_\Gamma^2-\frac{2\lambda^2_\sigma \Lambda^{-1}_\sigma}{\xi_t})|Z_t|^2 dt, t\in (0,T).
\end {eqnarray*}
 Then
\begin {eqnarray}
 \label {lp-esti}\frac{|Z_t|^2}{\xi_t}
&\leq&\frac{|x-y|^2}{\xi_0}+\int_0^t\frac{2}{\xi_r}\hat{\sigma}(r)^* Z_r\cdot d B_r+\int_0^t\frac{2|Z_r|^2}{\xi^{2}_r}(\frac{L}{2}\xi_r-\lambda_\sigma-\frac{1}{2}\xi'_r)dr,\\
 \label {lp-Esti}\frac{|Z_t|^2}{\xi_t}
&\leq&\frac{|x-y|^2}{\xi_0}+\int_0^t\frac{2}{\xi_r}\hat{\sigma}(r)^* Z_r\cdot d \tilde{B}_r+\int_0^t\frac{2|Z_r|^2}{\xi^{2}_r}(\frac{L}{2}\xi_r-\frac{\lambda^2_\sigma}{\Lambda_\sigma}-\frac{1}{2}\xi'_r)dr.
\end {eqnarray}
\subsection {Extension of $Y$ to $T$}
By the definition of $\xi$, we have $\frac{L}{2}\xi_r-\lambda_\sigma-\frac{1}{2}\xi'_r\leq \frac{L}{2}\xi_r-\frac{\lambda^2_\sigma}{\Lambda_\sigma}-\frac{1}{2}\xi'_r=-\theta$. So
\[\mathbb{E}[\int_0^t\frac{|Z_r|^2}{\xi^{2}_r}dr]\leq\frac{|x-y|^2}{2\theta\xi_0},\] which, however, does NOT imply
that $\frac{Z_r}{\xi_r}\in [M^2_G(0,T)]^d$. This is different from the linear case.

 Similar to the arguments in Section \ref {sec-l-1.2}, we can show that, for any $p\geq1$, there exists $C(p)>0$ such that
    \[\mathbb{E}[\int_0^t\frac{|Z_r|^{p}}{\xi^{p}_r}]dr\leq C(p), \ \textmd{for any} \ t\in[0,T).\]
Note  that, for any $t\in(0,T)$  and $p\geq 1$, $\frac{Z_r}{\xi_r}1_{[0,t](r)}\in [M^p_G(0,T)]^d$, and that
\[\mathbb{E}[\int_s^t\frac{|Z_r|^{p}}{\xi^{p}_r}]dr\leq C(\alpha p)^{\frac{1}{\alpha}}(t-s)^{\frac{1}{\alpha^*}} \
\textmd{for any}  \ \alpha>1.\] So there exists
\begin {eqnarray}\label {Drift}g\in [M^p_G(0,T)]^d \ \textmd{such that}  \ g_s=\frac{1}{\xi_s}(X_s-Y_s), \ s\in[0,T).
 \end {eqnarray}

 Let $(\bar{Y}_t)_{t\in[0,T]}$ be the solution to the equation below
\begin {eqnarray}
                 \label {equYe}d \bar{Y}_t&=&\sigma(\bar{Y}_t)d B_t +b(\bar{Y}_t)dt+\sigma(X_t)g_t dt, \  \bar{Y}_0=y.
\end {eqnarray} Clearly, we have $Y_t=\bar{Y}_t$, $t\in[0,T)$. So $\bar{Y}$ is a continuous extension of $Y$ to $[0,T]$. In the sequel, we shall still write $Y$ for $\bar{Y}$.

The following result is the counterpart of Proposition \ref {extention} in the $G$-expectation space.
\begin {proposition} Let $X$ and $Y$ be the solutions to equations (\ref {equX}) and (\ref {equYe}). We have \[X_T=Y_T, \ q.s.\]
\end {proposition}

\subsection {Girsanov Transformation with Bounded Drifts}
For a bounded process $h\in [M_G^2(0,T)]^d$, set
$\tilde{B}_t:=B_t-\int_0^th_sds$. We try to find a sublinear
expectation under which $\tilde{B}_t$ is a $G$-Brownian motion. Hu
et al (2014) constructed a sublinear expectation
$\tilde{\mathbb{E}}$ with such property  in an extended
$\hat{G}$-expectation space $(\hat{\Omega}_T, L^1_{\hat{G}},
\mathbb{E})$ with $\hat{\Omega}_T:=C_0([0,T]; \mathbb{R}^{2d})$ and \[
\hat{G}(A)=\frac{1}{2}\sup_{\gamma\in\Gamma}\mathrm{tr}\left[
A\left[
\begin{array}
[c]{cc}%
\gamma & I_d\\
I_d & \gamma^{-1}%
\end{array}
\right]  \right]  ,\ A\in\mathbb{S}_{2d}.
\] Let $\hat{B}_t=(B_{t},B'_{t})$ be the canonical process in the extended space. By the definition of $\hat{G}$, we have $\langle B,B'\rangle_t=t I_d$. \cite {HJPSb}
constructed the time consistent sublinear expectation
$\mathbb{\tilde{E}}_t$ in the following way
\[\mathbb{\tilde{E}}_t[\xi]= \mathbb{E}_t[\frac{U^h_T}{U^h_t}\xi], \ \textmd{for} \ \xi\in
L_{ip}(\Omega_T),\] where
\[ U^h_t:=\exp\{\int_{0}^{t}h_{s}\cdot dB'_{s}-\frac{1}{2}\int_{0}^{t}\mathrm{tr}[h_sh^*_s
d\langle B'\rangle_{s}]\}.\] \cite {HJPSb} proved that
$\tilde{B}_t:=B_t-\int_0^th_sds$ is a $G$-Brownian motion under
$\tilde{\mathbb{E}}$. For readers' convenience, we give a sketch of
the proof. Actually, it suffices to prove it for the case $h_s\equiv a\in \mathbb{R}^d$.
For any $\varphi \in C_b(\mathbb{R}^d)$, set
$u(t,x):=\tilde{\mathbb{E}}[\varphi(x+\tilde{B}_t)]$.We shall prove
that $u$ is the viscosity solution to the $G$-heat equation
\begin {eqnarray*}\partial_t u-G(D^2_x u)&=&0,\\
                               u(0,x)&=&\varphi(x).
\end {eqnarray*}
By the
definition of $\tilde{\mathbb{E}}$, we have
\begin
{eqnarray*}u(t+s,x)&=&\tilde{\mathbb{E}}[\varphi(x+\tilde{B}_{s+t})]\\
&=&\mathbb{E}[U^h_{t+s}\varphi(x+\tilde{B}_{t+s})]\\
&=&\mathbb{E}[U^h_s\mathbb{E}_s[\frac{U^h_{t+s}}{U^h_s}\varphi(x+\tilde{B}_{s}+\tilde{B}_{t+s}-\tilde{B}_{s})]]\\
&=&\mathbb{E}[U^h_su(t, x+\tilde{B}_s)]\\
&=&\mathbb{\tilde{E}}[u(t, x+\tilde{B}_s)].
\end {eqnarray*}
By It\^o's formula, we have $U^h_t\tilde{B}_t=\int_0^tU^h_rd B_r$. So
\[\mathbb{\tilde{E}}_s[\tilde{B}_t]=\mathbb{E}_s[\frac{U^h_t}{U^h_s}\tilde{B}_t]=\frac{1}{U^h_s}\int_0^sU^h_rd B_r=\tilde{B}_s,\] which implies that $\tilde{B}_t$ is a (symmetric) martingale under $\mathbb{\tilde{E}}$. Particularly, we have
\[\tilde{\mathbb{E}}[\tilde{B}_1]=0, \ \textmd{and} \ \frac{1}{2}\tilde{\mathbb{E}}[\langle A\tilde{B}_1, \tilde{B}_1\rangle]=\frac{1}{2}\tilde{\mathbb{E}}[\mathrm{tr}[ A\langle B\rangle_1]], \ \textmd{for} \ A\in \mathbb{S}_d.\]
On the one hand, we have $\frac{1}{2}\mathrm{tr}[ A\langle B\rangle_1]\leq G(A)$, and consequently
\[\frac{1}{2}\tilde{\mathbb{E}}[\langle A\tilde{B}_1, \tilde{B}_1\rangle]\leq G(A).\]
On the other hand, by the representation of the $G$-expectation (see \cite{DHP11}), we have
\[\frac{1}{2}\tilde{\mathbb{E}}[\langle A\tilde{B}_1, \tilde{B}_1\rangle]\geq \frac{1}{2}\sup_{\gamma\in \Gamma}\mathbb{E}_{P_\gamma}[U^h_T\mathrm{tr}[ A\langle B\rangle_1]]=\frac{1}{2}\sup_{\gamma\in \Gamma}\mathrm{tr}[A\gamma]=G(A),\] where
$P_\gamma$ is a probability on $\hat{\Omega}_T$ such that the canonical process $\hat{B}_t$ is a martingale with
\[
\langle \hat{B}\rangle_t=\left[
\begin{array}
[c]{cc}%
\gamma & I_d\\
I_d & \gamma^{-1}%
\end{array}
\right]t .
\]
Combining the above arguments, we can prove that $u$ is the
viscosity solution to the $G$-heat equation.

\subsection {Localizations and Estimates}
\subsubsection {Localizations}
To apply  localization procedure, one has to show that the corresponding stopping times are quasi-continuous, which is not obvious in the $G$-expectation space. In this section, we shall prove the quasi-continuity of hitting times for  processes of certain forms. The following result generalizes Theorem 4.1 in \cite {Song11b}.
\begin {lemma}\label {HT} Let $X_t=\int_0^t Z_s\cdot dB_s+\int_0^t\eta_s ds+\int_0^t \mathrm{tr}[\zeta_sd\langle B\rangle_s]$ with $Z\in [H^1_G(0,T)]^d$ and $\eta, \zeta^{i,j}\in M^1_G(0,T)$. Assume $\int_0^t\eta_s ds+\int_0^t \mathrm{tr}[\zeta_sd\langle B\rangle_s]$ is non-decreasing and \[\int_0^t \mathrm{tr}[Z_sZ_s^*d\langle B\rangle_s]+\int_0^t\eta_s ds+\int_0^t \mathrm{tr}[\zeta_sd\langle B\rangle_s]\] is strictly increasing. For $a>0$,  $\tau_a:= \inf\{t\geq 0| \ X_t>a\}\wedge T$ is quasi-continuous.
\end {lemma}
\begin {proof} Set $\underline{\tau}_a:= \inf\{t\geq 0| \ X_t\geq a\}\wedge T$. By Lemma 3.3 in \cite {Song11b}, it suffices to show that $[\tau_a>\underline{\tau}_a]$ is a polar set. Define
\[\mathcal{S}_a(X)=\{\omega\in \Omega_T| \ \textmd{there exists} \ (r, s)\in \mathcal{Q}_T \ s.t. \ X_t(\omega) = a \ \textmd{for all} \ t \in [s, r ]\},\]
where \[ \mathcal{Q}_T=\{(r, s)| \ T \geq r>s \geq 0, \ r, s \ \textmd{are rational}\}.\] It is clear that
\[[\tau_a>\underline{\tau}_a]\subset{\cal S}_a(X)\bigcup\cup_{r\in
\mathbb{Q}\cap[0, T]}[X_{r\wedge \tau_a}<X_{r\wedge
\underline{\tau}_a}],\] where $\mathbb{Q}$ denotes the totality of
rational numbers. By the assumption, we know that  ${\cal S}_a(X)$ is a polar set. Noting that $X_{r\wedge \tau_a}\leq X_{r\wedge
\underline{\tau}_a}$ and $\mathbb{E}[X_{r\wedge
\underline{\tau}_a}-X_{r\wedge \tau_a}]\leq0$, we conclude that $\cup_{r\in
\mathbb{Q}\cap[0, T]}[X_{r\wedge \tau_a}<X_{r\wedge
\underline{\tau}_a}]$ is also a polar set.
\end {proof}
\begin {corollary}Let $X_t=\int_0^t Z_s dB_s+\int_0^t\eta_s ds$ with $Z^{i,j}\in H^2_G(0,T)$ and $\eta^i\in M^2_G(0,T)$. Assume  $\int_0^t \mathrm{tr}[Z_s^*Z_sd\langle B\rangle_s]$ is strictly increasing. Then there exists a sequence of quasi-continuous stopping times $\tau_n$ such that $\tau_n\uparrow T$ and $(X_{t\wedge\tau_n})_{t\in[0,T]}$ is bounded.
\end {corollary}
\begin {proof} Applying It\^o's formula to $|X_t|^2$, we have \[|X_t|^2=\int_0^t2X_s^*Z_sdB_s+\int_0^t2X_s\cdot\eta_sds+\int_0^t \mathrm{tr}[Z_s^*Z_sd\langle B\rangle_s]. \] Set $Y_t:=\int_0^t2X_s^*Z_sdB_s+\int_0^t2|X_s\cdot\eta_s|ds+\int_0^t \mathrm{tr}[Z_s^*Z_sd\langle B\rangle_s]$ and $\tau_n:=\inf\{t\geq0| \ Y_t>n\}\wedge T$. By Lemma \ref {HT} we get the desired result.
\end {proof}
\subsubsection{Estimates}

Set $h_s=-\sigma(Y_s)^{-1}\sigma(X_s)g_s$,where $g$ is defined in (\ref {Drift}). Choose a sequence of quasi-continuous stopping times $\tau_n$ such that $\tau_n\uparrow T$, $\tau_n\leq T-\frac{1}{n}$ and $h^n_s:=h_s1_{[0,\tau_n]}$ is  bounded. Set  $\tilde{B}^n_t:=B_t-\int_0^th^n_sds$ and $\tilde{\mathbb{E}}^n[\xi]=\mathbb{E}[U^{h^n}_T\xi]$ for $\xi\in L_{ip}(\Omega_T)$.

The result below is the counterpart of Proposition \ref {esti-drift-l} in the $G$-expectation space.
\begin {proposition}\label {Estimate} \[\widetilde{\mathbb{E}}^{h^n}[\exp\{\frac{\theta^2}{8\Lambda_\sigma^2\Lambda_\Gamma^2}\int_0^{\tau_n}\frac{|Z_s|^2}{\xi_s^2}ds\}]
\leq \exp\{\frac{\theta}{8\Lambda_\sigma^2\Lambda_\Gamma^2\xi_0}|x-y|^2\}.\]
\[\mathbb{E}[| U^{h}_{\tau_n}|^{1+\delta}]\leq \exp\{\frac{\theta\sqrt{1+\delta^{-1}}}{8\Lambda_\sigma^2\Lambda_\Gamma^2\xi_0(1+\sqrt{1+\delta^{-1}})}|x-y|^2\},\]
 where $\delta:=\frac{\theta^2}{4\Lambda_\sigma^2\beta_\sigma^2\beta_\Gamma^2+4\theta \Lambda_\sigma\beta_\sigma\beta_\Gamma}$.
\end {proposition}

\begin {proof} By Girsanov transformation with bounded drifts, we know that $\tilde{B}^n_t$ is a $G$-Brownian motion under $\tilde{\mathbb{E}}^{n}$. Noting that $M_t:=\int_{0}^{t\wedge \tau_n} h_{s}\cdot dB'_{s}-\int_{0}^{t\wedge \tau_n} \mathrm{tr}[h_sh^*_s
d\langle B'\rangle_{s}]$ is a symmetric martingale under $\tilde{\mathbb{E}}^{n}$, the proof is similar to that of Proposition \ref {esti-drift-l}.
\end {proof}
\subsection {Girsanov Transformation with Unbounded Drifts}

 \begin {proposition} $\tilde{B}_t=B_t-\int_0^th_sds$ is a $G$-Brownian motion under $\tilde{\mathbb{E}}$.
 \end {proposition}
 \begin {proof}Noting that \[|e^x-e^y|\leq e^m|x-y|+e^{-\delta m}e^{(1+\delta)x}+e^{-\delta m}e^{(1+\delta)y},\] we have
 \begin {eqnarray*}& &\mathbb{E}[| U^h_T- U^{h}_{\tau_n}|]
 \leq e^m \mathbb{E}[|u^h_T -u^{h}_{\tau_n}|]+e^{-m\delta}\mathbb{E}[| U^{h}_{\tau_n}|^{(1+\delta)}]+e^{-m\delta}\mathbb{E}[| U^{h}_T|^{(1+\delta)}],
 \end {eqnarray*} where $u^h_t=\int_0^th_s dB'_s-\frac{1}{2}\int_0^t\mathrm{tr}[h_sh^*_s d\langle B'\rangle_s]$. By the Monotone Convergence Theorem under $G$-expectation (Theorem 31, Denis, et al (2011)), we have \[\lim_{n\rightarrow\infty}\mathbb{E}[\int_0^T|h_s-h^n_s|^2ds]=0.\] So we conclude that \[\lim_{n\rightarrow\infty}\mathbb{E}[|u^h_T -u^{h}_{\tau_n}|]=0.\] First letting $n$ go to infinity, then letting $m$ go to infinity, we get that \[\mathbb{E}[| U^h_T- U^{h}_{\tau_n}|]\rightarrow0\] by Proposition \ref {Estimate}.

 For a function $\varphi\in C_{b,Lip}(\mathbb{R}^k)$ and a partition of $[0,T]$: $0\leq t_1 <t_2<\cdot\cdot\cdot<t_k\leq T$, set $\xi=\varphi(\tilde{B}_{t_1}, \cdot\cdot\cdot, \tilde{B}_{t_k})$ and $\xi^n=\varphi(\tilde{B}^n_{t_1}, \cdot\cdot\cdot, \tilde{B}^n_{t_k})$.
\begin {eqnarray*}|\mathbb{E}[ U^h_T\xi]-\mathbb{E}[ U^{h^n}_T\xi^n]|\leq|\mathbb{E}[ U^h_T\xi]-\mathbb{E}[ U_T^{h^n}\xi]|
+|\mathbb{E}[ U^{h^n}_T\xi]-\mathbb{E}[ U^{h^n}_T\xi^n]|\rightarrow0.
\end {eqnarray*} By Girsanov transformation with bounded drifts,  $\mathbb{E}[ U^{h^n}_T\xi^n]=\mathbb{E}[\varphi(B_{t_1}, \cdot\cdot\cdot, B_{t_k})]$ for each $n$. By this, we get the desired result.
\end {proof}

\begin {corollary}$ \label {coupling} \mathbb{E}[\varphi(X^y_T)]=\tilde{\mathbb{E}}[\varphi(Y^y_T)]$ for any $\varphi\in C_b(\mathbb{R}^d)$.
\end {corollary}
\subsection {Proof to the Gradient Estimates}

\textbf{Proof to Theorem \ref {main2}.} (\ref {lp-esti}-\ref {lp-Esti}) shows that
\begin {eqnarray*}
\int_0^t\frac{|Z_r|^2}{\xi_r^2}dr &\leq& \frac{|x-y|^2}{2\theta'\xi_0}+\int_0^t\frac{1}{\theta'\xi_r}\hat{\sigma}(r)^* Z_r\cdot d B_r,\\
\int_0^t\frac{|Z_r|^2}{\xi_r^2}dr &\leq& \frac{|x-y|^2}{2\theta\xi_0}+\int_0^t\frac{1}{\theta\xi_r}\hat{\sigma}(r)^* Z_r\cdot d \tilde{B}_r,
\end {eqnarray*} where $\theta'=\theta+\lambda_\sigma-\lambda_\sigma^2\Lambda_\sigma^{-1}$. So
\begin {eqnarray*}
\mathbb{E}[\int_0^T \mathrm{tr}[h_sh_s^*d\langle B'\rangle_s]]\leq \frac{ \beta_\sigma^2}{2\lambda_\Gamma^2\theta'\xi_0}|x-y|^2,\\
\tilde{\mathbb{E}}[\int_0^T \mathrm{tr}[h_sh_s^*d\langle B'\rangle_s]]\leq \frac{ \beta_\sigma^2}{2\lambda_\Gamma^2\theta\xi_0}|x-y|^2.
\end {eqnarray*} By Corollary \ref {coupling},
\begin {eqnarray*}& &|\mathbb{E}[\varphi(X^y_T)]-E[\varphi(X^x_T)]|\\
&=&|\tilde{\mathbb{E}}[\varphi(X^x_T)]-\mathbb{E}[\varphi(X^x_T)]| \\
&\leq& \|\varphi\|_\infty \mathbb{E}[| U^h_T-1|]\\
&\leq& \|\varphi\|_\infty (\tilde{\mathbb{E}}[|u^h_T|]+\mathbb{E}[|u^h_T|])\\
&\leq& 2\|\varphi\|_\infty( \frac{ \beta_\sigma^2}{4\lambda_\Gamma^2\theta\xi_0}|x-y|^2+ \sqrt{\frac{\beta_\sigma^2}{2\lambda_\Gamma^2\theta\xi_0}}|x-y|).
\end {eqnarray*}
 So, by Lemma \ref {observation},
\[|\mathbb{E}[\varphi(X^y_T)]-\mathbb{E}[\varphi(X^x_T)]|\leq \sqrt{\frac{2\beta_\sigma^2}{\lambda_\Gamma^2 \theta \xi_0}} \|\varphi\|_\infty|x-y|.\]

 Taking $\theta=\frac{\lambda_\sigma^2\Lambda_\sigma^{-1}}{2}$, we get
\[|\mathbb{E}[\varphi(X^y_T)]-\mathbb{E}[\varphi(X^x_T)]|\leq \frac{2\Lambda_\sigma^2}{\lambda_\sigma^3\lambda_\Gamma}\frac{\|\varphi\|_\infty}{\sqrt{(1-e^{-LT})/L}}|x-y|. \]
$\square$

%%% ----------------------------------------------------------------------

%%%%%%%%%%%%%%%%%%%%%%%≤ŒøºŒƒœ◊
\renewcommand{\refname}{\large References}{\normalsize \ }

\end{document}